\documentclass[12pt]{amsart}
\usepackage{amssymb}
\usepackage{ljm-auth}
\usepackage{graphicx}

\newtheorem{definition}{Definition}

\newtheorem{proposition}{Proposition}
\newtheorem{remark}{Remark}
\newtheorem{example}{Example}
\newtheorem{corollary}{Corollary}
\newtheorem{lemma}{Lemma}

\author{O. Dovgoshey, J. Riihentaus}
\crauthor{Dovgoshey, Riihentaus} 

\tit{On quasinearly subharmonic functions}
\shorttit{On quasinearly subharmonic functions} 

\begin{document}

\noindent\emph{This paper is dedicated to Professor Vladimir Gutlyanskii on the occasion of his 75-th anniversary.}

\medskip

\maketit

\address{Institute of Applied Mathematics and Mechanics of the NASU, Dobrovolskogo 1,   Sloviansk, 84100, Ukraine\\  Department of Mathematical Sciences, University of Oulu, P. O. Box~3000, FI $-$ 90014 Oulun yliopisto, Finland; Department of Physics and Mathematics, University of Eastern Finland, P. O. Box~111, FI $-$ 80101 Joensuu, Finland}

\email{aleksdov@mail.ru, juhani.riihentaus@gmail.com}

\abstract{We recall the definition of quasinearly subharmonic functions, point  out that this function class includes, among others, subharmonic functions,
quasisubharmonic functions, nearly subharmonic functions and essentially almost subharmonic functions. It is shown that the sum of two quasinearly subharmonic functions may not be quasinearly subharmonic. Moreover, we characterize the harmonicity via quasinearly subharmonicity.}

\notes{0}{
\thank{The publication is based on the research provided by the grant support of the State Fund For Fundamental Research (project N $20570$). The first author was also partially supported by Project 15-1bb$\setminus19$ ``Metric Spaces, Harmonic Analysis of Functions and Operators and Singular and Nonclassic Problems for Differential Equations'' (Donetsk National University, Vinnitsia, Ukraine)}}

\section{Subharmonic functions and nearly subharmonic functions.}

Denote by $\mathbb R^N$ the $N$-dimensional Euclidean space. If $x \in \mathbb R^N$, then the open ball centered at $x$ with radius $r>0$ will be denoted by $B^N(x,r)$ and we will write $\overline{B^N(x,r)}$ for the closure of this ball.

Let $D$ be a domain in $\mathbb R^N$, $N \geq 2$. An upper semicontinuous function $u:\, D\rightarrow [-\infty ,+\infty )$ is \emph{subharmonic} if the inequality
\[
u(x)\leq \frac{1}{\nu _N\, r^N}\int\limits_{B^N(x,r)}u(y)\, dm_N(y)
\]
holds for all $\overline{B^N(x,r)}\subset D$, where $\nu _N$ is the volume of the unit ball in $\mathbb R^N$.

The function $u\equiv -\infty $  is considered  subharmonic. A function $u$ defined on an open set $\Omega \subseteq \mathbb R^N$ is subharmonic if the restriction of $u$ to arbitrary connected component of $\Omega$ is subharmonic.

\begin{definition}\label{def1.1}
A function $u:\, D\rightarrow [-\infty ,+\infty )$ is \emph{nearly subharmonic}, if $u$ is Lebesgue measurable, $u^+\in {\mathcal{L}}^1_{\mathrm{loc}}(D)$
and
\begin{equation}\label{eq1.1}
u(x)\leq \frac{1}{\nu _N\, r^N}\int\limits_{B^N(x,r)}u(y)\, dm_N(y)
\end{equation}
holds for all $\overline{B^N(x,r)}\subset D$.
\end{definition}
Observe that our definition is slightly nonstandard because in the standard definition of nearly subharmonic functions one uses the stronger assumption $u\in {\mathcal{L}}^1_{\textrm{loc}}(D)$, see e.g. \cite{Her71}, p.~14.

The following lemma is an analog of Proposition~2.2 (vii) from \cite{Rii07$_4$}, p.~55, and Proposition~1.5.2 (vii) from \cite{Rii09$_2$}, p.~e2615.
\begin{lemma}\label{lem1.2}
Let $D$ be a domain in ${\mathbb{R}}^N$, $N\geq 2$, and let $u: D\rightarrow [-\infty ,+\infty)$ be nearly subharmonic in the sense of Definition~\ref{def1.1}. Then either $u \in \mathcal L_{\mathrm{loc}}^1(D)$ or the equality $u(x)=-\infty$ holds for every $x \in D$.
\end{lemma}
\begin{proof}
Suppose $u \notin \mathcal L_{\mathrm{loc}}^1(D)$. Then there is a compact set $K \subset D$ such that
\begin{equation}\label{eq1.2}
\int_K u(y)\,dm_N(y) = -\infty.
\end{equation}
Since $K$ is compact and $D$ is open, we have
\[
\textrm{dist}(K, \partial D) = \inf_{x \in K,\ y \in \partial D} |x-y|>0.
\]
Let $\varepsilon$ be a positive real number satisfying the inequality
\begin{equation}\label{eq1.3}
3\varepsilon < \textrm{dist}(K, \partial D).
\end{equation}
We can find a finite set of balls $B^N(x_1, \varepsilon)$, $\ldots$, $B^N(x_m, \varepsilon)$ such that $x_i \in K$ for every $i\in \{1,\ldots, m\}$ and
\[
K \subseteq \bigcup_{i=1}^m B^N(x_i, \varepsilon) \subseteq D.
\]
These inclusions and \eqref{eq1.2} imply
\begin{equation}\label{eq1.4}
\int_{B^N(x_{i_0}, \varepsilon)} u(y)\,dm_N(y) = -\infty
\end{equation}
for some $i_{0}\in\{1, ..., m\}.$
It follows from \eqref{eq1.3} and $x_{i_0}\in K$ that $$B^N(x_{i_0}, \varepsilon) \subset B^N(x, 2\varepsilon) \subseteq D$$ holds for every $x \in B^{N}(x_{i_0}, \varepsilon)$. Using~\eqref{eq1.4} we obtain
\[
\int_{B^N(x, 2\varepsilon)} u(y)\,dm_N(y) = -\infty
\]
for every $x \in B^N(x_{i_0}, \varepsilon)$. Since $u$ is nearly subharmonic, it follows that
\[
-\infty \leq u(x) \leq \frac{1}{\nu _N (2\varepsilon)^{N}}\int_{B^N(x, 2\varepsilon)} u(y)\,dm_N(y) = -\infty,
\]
i.e., $u(x)=-\infty$ for every $x \in B^{N}(x_{i_0}, \varepsilon)$. Write
\[
A = \{x \in D: u(x) = -\infty\}.
\]
Since $B^N(x_{i_0}, \varepsilon) \subseteq A$, the interior of $A$ is non-void, $\textrm{Int}(A) \neq \varnothing$. To complete the proof, it is sufficient to show that $\textrm{Int}(A) = D$. If the last equality does not hold, then there is a point $y^*\in D\cap \partial \textrm{Int}(A)$. Let $0<\delta^*<\frac{1}{2}\textrm{dist}(y^*, \partial D)$. Then for every $y\in B^{N}(y^*, \delta^*)$ we have
\[
D \supseteq B^N(y, 2\delta^*) \text{ and } B^N(y, 2\delta^*) \cap \textrm{Int}(A) \neq \varnothing.
\]
Consequently $u(y)=-\infty$ holds for every $y \in B^N(y^*, \delta^*)$. Thus $y^* \in \textrm{Int}(A)$, contrary to $y^* \in \partial \textrm{Int}(A)$.
\end{proof}


The following proposition is well known under the additional condition $u \in \mathfrak{L}_{\mathrm{loc}}^{1}(D)$.

\begin{proposition}\label{pr1.3}
Let $D$ be a domain in ${\mathbb{R}}^N$, $N\geq 2$ and let $u:\,D\rightarrow [-\infty ,+\infty )$ be Lebesgue measurable.
Then  $u$ is nearly subharmonic in $D$ if and only if there exists a subharmonic in $D$ function $u^*$ such that $u^*(x)\geq u(x)$ for all $x \in D$ and $u^*(x)=u(x)$ holds Lebesgue almost everywhere.
\end{proposition}
\begin{proof}
If $u(x)\equiv -\infty$ then, the proposition is evident. In the opposite case by Lemma~\ref{lem1.2} we have $u \in \mathcal L_{\mathrm{loc}}^1(D)$, and it is a reformulation of Theorem~1 from~\cite{Her71}, p.~14.
\end{proof}

\begin{remark}\label{r1.4}
In particular, if $u$ is nearly subharmonic, then we can take $u^*$ as the lowest  upper semicontinuous majorant of $u$:
\[
u^*(x)=\limsup_{x'\rightarrow x}\ u(x').
\]
\end{remark}

Observe also that the \textit{almost subharmonic functions}, by Szpilrajn \cite{Sz33}, are included in Definition~\ref{def1.1} in the following sense. Let  $u: \, D\rightarrow [-\infty ,+\infty )$ be almost subharmonic, that is, $u\in {\mathcal{L}}^1_{\textrm{loc}}(D)$ and inequality~\eqref{eq1.1} is satisfied for Lebesgue almost every $x\in D$ with all $\overline{B^N(x,r)}\subset D$ . Let
\[
D_1:=\{x\in D\,:\, u(x)\leq \frac{1}{\nu _N\, r^N}\int\limits_{B^N(x,r)}u(y)\, dm_N(y)\, {\textrm{ for all }}\overline{B^N(x,r)}\subset D\,   \}.
\]
Define $\tilde {u}:\, D\rightarrow [-\infty ,+\infty )$ as
\[
\tilde {u}(x):=\begin{cases}u(x), &{  \textrm{ when }}x\in D_1,\\
-\infty , &{  \textrm{ when }}x\in D\setminus D_1.\end{cases}
\]
By assumption $m_N(D\setminus D_1)=0$, it is easy to see that $\tilde {u}$ is nearly subharmonic in $D$.

In the connection with almost subharmonic functions see also \cite{CoRa97} and \cite{Ra37}, p.~20, and  \cite{LiLo01}, p.~238. Lieb and Loss even call the almost subharmonic functions briefly subharmonic functions.

\section{Quasinearly subharmonic functions}

Let $D$ be a domain in ${\mathbb{R}^N}$, $N\geq 2$. It is an important fact that if  $u: D\rightarrow [0,+\infty )$ is subharmonic and $p>0$, then  there exists a constant $K=K(N,p)>0$ such that the inequality
\begin{equation}\label{eq1.5}
u(x)^p\leq \frac{K}{\nu _N\, r^N}\int\limits_{B^N(x,r)}u(y)^p\, dm_N(y)
\end{equation}
holds for all $\overline{B^N(x,r)}\subset D$. In the case of $p=1$ and $K=1$, inequlity \eqref{eq1.5} is just the familiar mean value inequality for (nonnegative) subharmonic functions. The case $p>1$ follows from the the case $p=1$ with the aid of Jensen's inequality. The case $0<p<1$ has been given in Fefferman and Stein \cite{FeSt72}, Lemma~2, p.~172  and in \cite{Ku74}, Theorem~1, p.~529, where, however, only absolute values of harmonic functions were considered. The proofs in \cite{FeSt72} and in \cite{Ku74} are somewhat long. See also \cite{Ga07}, Lemma~3.7, p.~116, and  \cite{AhRu93}, (1.5), p.~210. In \cite{Rii89}, Lemma, p.~69, it was pointed out that the proof in \cite{FeSt72} includes the case of nonnegative subharmonic functions, too. See also \cite{Su90}, p. 271, \cite{Su91}, p. 114, \cite{Ha92}, Lemma~1, p.~113,  \cite{St98}, Lemma~3, p.~305, \cite{St02}, p.~794,  \cite{St04}, \cite{St13}, Lemma~1, p.~363, \cite{St14}, Lemma~2.1, p.~7, \cite{DjPa07$_2$}, Theorem~A, p.~413, and \cite{AhBr88}, p.~132. Observe that a possibility for an essentially different proof was pointed out already in \cite{To86}, pp.~188-190. Later other different proofs were given in \cite{Pa94}, p.~18 and Theorem~1, p.~19 (see also \cite{Pa96$_2$}, Theorem~A, p.~15), \cite{Rii00}, pp.~233-234, \cite{Rii01}, p.~188. The results in \cite{Pa94}, \cite{Rii00} and \cite{Rii01} hold in fact for more general function classes than just for nonnegative subharmonic functions. Compare also \cite{DiTr84}, \cite{Do57}, p.~430, and \cite{Do88}, p.~485.

Inequality \eqref{eq1.5} has many applications. Among others, it has been applied to the weighted boundary behavior of subharmonic functions and to the nonintegrability of subharmonic or superharmonic functions.

It is therefore relevant to find a generalization of results related to inequality~\eqref{eq1.5}. We will do this in the following way.

Let $D$ be a domain in $\mathbb R^N$, $N \geq 2$. For every $u:D\to [-\infty, +\infty)$ and $M\ge 0$ we write $u_{M}:=\max\{u, -M\}+M.$

\begin{definition}\label{def1.4}
Let $K\in[1, +\infty).$ A Lebesgue measurable function $u:\,D \rightarrow [-\infty ,+\infty )$ is \emph{$K$-quasinearly subharmonic}, if  $u^+\in{\mathcal{L}}^{1}_{\mathrm{loc}}(D)$ and the inequality
\begin{equation*}
u_M(x)\leq \frac{K}{\nu _N\,r^N}\int\limits_{B^N(x,r)}u_M(y)\, dm_N(y)
\end{equation*}
holds for all $M\geq 0$ and $\overline{B^N(x,r)}\subset D.$ A function $u:\, D\rightarrow [-\infty ,+\infty )$ is \emph{quasinearly subharmonic}, if $u$ is
$K$-quasinearly subharmonic for some $K$.
\end{definition}

In addition to the above defined class of quasinearly subharmonic functions, we will consider their proper subclass.

\begin{definition}\label{def1.5}
A~Lebesgue  measurable function $u:\,D \rightarrow [-\infty ,+\infty )$ is \emph{$K$-quasinearly subharmonic n.s. (in the narrow sense)}, if $u^+\in{\mathcal{L}}^{1}_{\mathrm{loc}}(D)$ and if there is a constant $K=K(N,u,D)\geq 1$ such that the inequality
\begin{equation*}
u(x)\leq \frac{K}{\nu _N\,r^N}\int\limits_{B^N(x,r)}u(y)\, dm_N(y)
\end{equation*}
holds for all $\overline{B^N(x,r)}\subset D$.
A function $u:\, D\rightarrow [-\infty ,+\infty )$ is \emph{quasinearly subharmonic n.s.}, if $u$ is $K$-quasinearly subharmonic n.s. in $D$ for some $K$.
\end{definition}

For a function $u$ is defined on an open set $\Omega \subseteq \mathbb R^n$, the quasinearly subharmonicity (quasinearly subharmonicity n.s.) of $u$ means that the restriction of $u$ to arbitrary connected component of $\Omega$ is quasinearly subharmonic (quasinearly subharmonic n.s.).

Observe that if $u:\,D \rightarrow [0,+\infty )$ is  subharmonic and $p>0$, then $u^p$ is quasinearly subharmonic n.s. and thus also quasinearly subharmonic, see statement (1) and statement (4) of Proposition~\ref{pr1.4} below and also~\cite{DoRi13}.

More generally, the class of quasinearly subharmonic functions includes, among others the subharmonic and nearly subharmonic functions and also the quasisubharmonic functions (for the definition of this see \cite{Rii07$_4$} and \cite{Her71}), also functions satisfying certain natural  growth conditions, especially certain eigenfunctions,  polyharmonic functions, subsolutions of certain general elliptic equations.

Let $D$ be a domain in $\mathbb R^{N},$ $N\ge 2.$ Recall that a continuous function $u: D\to[0, \infty)$ is said to be a Harnack function if there are $\lambda\in(0, 1)$ and $C_{\lambda}\in[1, \infty)$ such that the following Harnack inequality
$$\max\limits_{z\in B^{N}(x, \lambda r)}u(z)\le C_{\lambda}\min\limits_{z\in B^{N}(x, \lambda r)}u(z)$$
holds whenever $B^{N}(x, r)\subseteq D.$ See \cite{Vu82}, p.~259.
Every Harnack function is quasinearly subharmonic. This implies the quasinearly subharmonicity of nonnegative harmonic functions as well as nonnegative solutions of some elliptic equations. In particular, the partial differential equations associated with quasiregular mappings belong to this family of elliptic equations, see Vuorinen \cite{Vu82} and the above references.

Observe that already Domar  \cite{Do57} has pointed out the relevance of the class of (nonnegative) quasinearly subharmonic functions. For, at least partly, more general function class, see \cite{Do88}.


We list below four simple examples of quasinearly subharmonic functions.
\begin{example}\label{ex1}
Let $D$ be a domain in ${\mathbb{R}}^N$, $N\geq 2$. Any Lebesgue measurable function   $u:\,D\rightarrow [m,M]$,
where $0<m\leq M<+\infty $, is quasinearly subharmonic n.s. and, because of Proposition~\ref{pr1.4} (see below), also quasinearly subharmonic. If $u$ is moreover
continuous, then $u$ is a Harnack function.
\end{example}

\begin{example}\label{ex2}
The function  $u:\,{\mathbb{R}}^2\rightarrow {\mathbb{R}}$,
\begin{displaymath}
u(x,y):=\begin{cases}-1, & {\textrm{when }}\, y<0\\
1, & {\textrm{when }}\, y\geq 0,\end{cases}
\end{displaymath}
is $2$-quasinearly subharmonic, but not quasinearly subharmonic n.s..
\end{example}

\begin{example}\label{ex3}
Let $D=(0,2)\times (0,1)\subset {\mathbb{R}}^2$ and let $c<0$ be  arbitrary. Let $E\subset D$ be a Borel set of zero Lebesgue measure. Let
$u:\,D\rightarrow [-\infty ,+\infty )$,
\begin{displaymath}
u(x,y):=\begin{cases}c , & {\textrm{ when }}\, (x,y)\in E,\\
1, &{\textrm{ when }}\, (x,y)\in D\setminus E  {\textrm{ and }}\, 0<x<1,\\
2, & {\textrm{ when }}\, (x,y)\in D\setminus E {\textrm{ and }}\, 1\leq x<2.
\end{cases}
\end{displaymath}
The function $u$ attains both negative and positive values, is $2$-quasinearly subharmonic n.s, but not nearly subharmonic.
\end{example}

\begin{example}
Let $D$ be a domain in ${\mathbb{R}}^N$, $N\geq 2$, and let $u:\,D\rightarrow [-\infty ,+\infty )$ be a quasinearly
subharmonic function (resp. quasinearly subharmonic n.s.). Let $E\subset D$ be a Borel set of zero Lebesgue measure. Let
$v:\,D\rightarrow [-\infty ,+\infty )$,
\begin{displaymath}
v(x):=
\begin{cases}-\infty , & {\textrm{ when }}\, x\in E,\\
u(x), &{\textrm{ when }}\, x\in D\setminus E.
\end{cases}
\end{displaymath}
The function $v$ is quasinearly subharmonic (resp. quasinearly subharmonic n.s.).
\end{example}

The term quasinearly subharmonic function was first introduced in \cite{Rii04$_1$}. Quasinearly subharmonic functions (sometimes with a different terminology), or, essentially, perhaps just functions satisfying a certain  generalized mean value
inequality, have previously been considered, or used, in addition to the above listed references at least in  \cite{Miz96}, \cite{Rii04$_1$}, \cite{Rii04$_2$}, \cite{Rii05},  \cite{Rii07$_1$}, \cite{Rii07$_2$}, \cite{DjPa07$_1$}, \cite{Ko07}, \cite{Rii08$_2$}, \cite{Rii09$_1$}, \cite{DoRi10$_1$}, \cite{DoRi10$_2$}, \cite{KoMa12} and \cite{Mih13}.

\section{Basic properties of quasinearly subharmonic functions}

Recall that a function $\varphi :\,[0,+\infty )\rightarrow [0,+\infty )$ satisfies a $\varDelta_2$-\emph{condition}, if there is a constant $C=C(\varphi )\geq 1$
such that $\varphi (2t)\leq C\,\varphi (t)$ for all $t\in [0,+\infty )$.

\begin{definition}\label{def1.7}
A function $\psi :\, [0,+\infty )\rightarrow [0,+\infty )$ is \textit{permissible}, if there  exist an increasing (strictly or not), convex function $\psi_1 :\,[0,+\infty )\rightarrow [0,+\infty )$ and a strictly increasing surjection \mbox{$\psi_2 :\, [0,+\infty )\rightarrow [0,+\infty )$} such that $\psi =\psi_2\circ \psi_1$ and the following
conditions hold:
\begin{enumerate}
\renewcommand{\theenumi}{\alph{enumi}}
\item $\psi_1$  satisfies the $\varDelta_2$-condition,
\item $\psi_2^{-1}$ satisfies the $\varDelta_2$-condition,
\item the function $t\mapsto \frac{\psi_2(t)}{t}$ is \textit{quasi-decreasing}, i.e. there is a constant $C=C(\psi_2)>0$ such that
\[
\frac{\psi_2(s)}{s}\geq  C\, \frac{\psi_2(t)}{t}
\]
whenever $0< s\leq t$.
\end{enumerate}
\end{definition}

Permissible functions are necessarily continuous.

Examples of permissible functions are:  $\psi_1(t)=t^p$, $p>0$, and   $\psi_2(t)=c\, t^{p\alpha }[\log (\delta +t^{p\gamma })]^\beta $, $c>0$,
$0<\alpha <1$, $\delta \geq 1$, $\beta ,\gamma \in {\mathbb{R}}$ such that $0<\alpha +\beta \,\gamma <1$, and $p\geq 1$. And also functions of the form
$\psi _3=\phi \circ \varphi$, where $\phi :[0,+\infty )\rightarrow [0,+\infty )$ is a concave surjection whose inverse $\phi ^{-1}$ satisfies
the $\Delta_2$-condition and $\varphi :[0,+\infty )\rightarrow [0,+\infty )$ is an increasing, convex function satisfying the $\Delta_2$-condition.

It is interesting to note  the following fact, see \cite{PaRi08}, Lemma~1 and Remark~1, p.~93:

\emph{Let $\psi :\, [0,+\infty )\rightarrow [0,+\infty )$ be  a permissible function. Then}
\begin{enumerate}
\item \emph{there are a number  $p>0$ and a convex function $M:\,[0,+\infty )\rightarrow [0,+\infty )$ satisfying the $\varDelta_2$-condition such that
$\psi (t)\asymp g(t^p)$, that is, there exist constants $C_1>0$ and $C_2>0$ such that}
\[
C_1\leq \frac{\psi (t)}{g(t^p)}\leq C_2
\]
\emph{for all $t>0$};
\item \emph{there are a number  $p>0$ and a convex function $\vartheta :\,[0,+\infty )\rightarrow [0,+\infty )$ satisfying the $\varDelta_2$-condition  such that
$\psi (t)\asymp \vartheta (t)^p$.}
\end{enumerate}

Next we list certain basic  properties of quasinearly subharmonic functions, see \cite{Rii07$_4$}, Proposition~2.1  and Proposition~2.2 and \cite{Rii09$_2$}, Proposition~1.5.1.
\begin{proposition}\label{pr1.4}
Let $D$ be a domain in ${\mathbb{R}}^N$, $N\geq 2$.
\begin{enumerate}
\item If $u:\, D\rightarrow [0 ,+\infty )$ is Lebesgue measurable and $u^+\in\mathcal{L}^{1}_{\textrm{loc}}(D)$, then $u$ is  $K$-quasinearly subharmonic if and only if $u$ is $K$-quasinearly subharmonic n.s., that is, if
\[
u(x)\leq \frac{K}{\nu _N\, r^N}\int\limits_{B^N(x,r)}u(y)\, dm_N(y)
\]
holds for all $\overline{B^N(x,r)}\subset D$.
\item If  $u:\, D\rightarrow [-\infty  ,+\infty )$ is $K$-quasinearly subharmonic n.s., then $u$ is $K$-quasinearly subharmonic in $D$, but not necessarily conversely.

\item A function $u:\, D\rightarrow [-\infty ,+\infty )$ is $1$-quasinearly subharmonic if and only if it is nearly subharmonic, that is, $1$-quasinearly subharmonic n.s.

\item If $u:\, D\rightarrow [0,+\infty )$ is quasinearly subharmonic and $\psi :\, [0,+\infty )\rightarrow [0,+\infty )$ is permissible, then $\psi \circ u$ is quasinearly subharmonic in $D$. Especially, if $h: D\rightarrow {\mathbb{R}}$ is harmonic and $p>0$, then $\mid h\mid ^p$ is quasinearly subharmonic.
\item The Harnack functions are quasinearly subharmonic.
\end{enumerate}
\end{proposition}
\begin{proof}
We leave statements~(1), (2) and (5) to the reader. For the proof of statement (4), see \cite{Rii06$_2$}, Lemma 2.1, p.~32. To prove statement~(3) suppose that  $u$ is nearly subharmonic in $D$. Then clearly $u_M$ is nearly subharmonic for all $M\geq 0$, and thus for every $\overline{B^N(x,r)}\subset D$, one has
\begin{equation*}
u_M(x)\leq \frac{1}{\nu _N\,r^N}\int\limits_{B^N(x,r)}u_M(y)\, dm_N(y).
\end{equation*}
Hence $u$ is $1$-quasinearly subharmonic.

On the other hand, if $u$ is $1$-quasinearly subharmonic in $D$, then one sees at once, with the aid of the  Lebesgue Monotone Convergence Theorem, that
$u$ is nearly subharmonic in $D$.
\end{proof}

\begin{proposition}\label{pr1.5}
Let $D$ be a domain in ${\mathbb{R}}^N$, $N\geq 2$.
\begin{enumerate}
\item If  $u:\, D\rightarrow [-\infty ,+\infty )$   is $K_1$-quasinearly subharmonic and $K_2\geq K_1$, then $u$ is $K_2$-quasinearly subharmonic.

\item If $u_1:\, D\rightarrow [-\infty ,+\infty )$    and $u_2:\, D\rightarrow [-\infty ,+\infty )$   are $K$-quasinearly subharmonic n.s.,   then $\lambda_1 u_1+\lambda _2u_2$ is $K$-quasinearly subharmonic n.s. for all  $\lambda_1,\, \lambda_2   \geq 0$.

\item If  $u:\, D\rightarrow [-\infty ,+\infty )$ is quasinearly subharmonic, then $u$ is locally bounded above.

\item If $u_j:\, D\rightarrow [-\infty ,+\infty )$, $j=1,2,\dots$, are $K$-quasinearly subharmonic (resp. $K$-quasinearly subharmonic n.s.), and $u_j\searrow u$ as $j\rightarrow +\infty $, then $u$ is $K$-quasinearly subharmonic (resp. $K$-quasinearly subharmonic n.s.).

\item If $u:\, D\rightarrow [-\infty ,+\infty )$   is $K_1$-quasinearly subharmonic and $v:\, D\rightarrow [-\infty ,+\infty )$ is $K_2$-quasinearly subharmonic, then $\max\{u,v\}$ is $K$-quasine\-arly subharmonic in $D$ with $K=\max\{K_1,K_2\}$. Especially,  $u^+:=\max\{u,0\}$ is $K_1$-quasinearly subharmonic.

\item Let ${\mathcal{F}}$ be a family of  $K$-quasinearly subharmonic (resp. $K$-quasinearly subharmonic n.s.) functions in $D$ and let $w:=\sup_{u\in \mathcal{F}}u$. If $w$ is Lebesgue measurable and $w^+\in {\mathcal{L}}_{{\mathrm{loc}}}^1(D)$, then $w$ is $K$-quasinearly subharmonic (resp. $K$-quasinearly subharmonic n.s.).

\item If  $u:\, D\rightarrow [-\infty ,+\infty )$   is quasinearly subharmonic n.s., then either $u\equiv -\infty $ or $u$ is finite almost everywhere, and $u\in \mathcal{L}^1_{\mathrm{loc}}(D)$.
\end{enumerate}
\end{proposition}

We leave the simple statements~(1)--(6) to the reader and note only that a proof of (7) is completely similar to the proof of Lemma~\ref{lem1.2}.

\begin{remark}
If $u:D\to[-\infty, \infty)$ is strictly negative, finite and constant, then $u$ is nearly subharmonic but, for every $K>1,$ $u$ is not $K$-nearly subharmonic n.s. Thus, the analog of statement (1) does not hold for functions which are quasinearly subharmonic in narrow sense.
\end{remark}

\begin{remark}\label{r1.10}
Related to statement (2) above, it is easy to see that, if $u:\, D\rightarrow [-\infty ,+\infty )$  is $K$-quasinearly subharmonic, then $\lambda u+C$ is $K$-quasinearly subharmonic for all  $\lambda  \geq 0$ and $C\geq 0$.
\end{remark}

The following example shows that the sum of two quasinearly subharmonic functions can be not quasinearly subharmonic.

\begin{example}\label{ex5}
The function $u: \mathbb R^{2}\to\mathbb R$,
\begin{displaymath} u(x,y):=\begin{cases}3 , & {\textrm{when }}\, x = 0,\\
1, & {\textrm{when }}\, x\ne 0,\end{cases}\end{displaymath}
is $3$-quasinearly subharmonic. The constant function $v:\mathbb R^{2}\to\mathbb R,$ $$v(x,y)\equiv -2,$$ is harmonic. Then we have
\begin{displaymath} (u+v)(x,y):=\begin{cases}1 , & {\textrm{when }}\, x = 0\\
-1, & {\textrm{when }}\, x\ne 0\end{cases}\end{displaymath} and $$(u+v)_{M}=\max\{u+v, -M\}+M=(u+v+M)^{+}$$ for every $M\ge 0$.  In particular for $M=1$ we obtain \begin{displaymath} (u+v)_{1}(x,y):=\begin{cases}2 , & {\textrm{when }}\, x = 0\\
0, & {\textrm{when }}\, x\ne 0.\end{cases}\end{displaymath}
Since $(u+v)_{1}(0,0)>1$ and the double integral $\iint\limits_{B}(u+v)_{1}(x, y)dxdy$ is zero for every ball $B\subset\mathbb R^{2}$, the function $(u+v)_{1}$ is not quasinearly subharmonic. Hence $(u+v)$ is also not quasinearly subharmonic.
\end{example}

\begin{remark}\label{r1.11}
It is easy to see that the analog of statement (7) from Proposition~\ref{pr1.5} does not hold for quasinearly subharmonic functions. A counterexample is the function $u:\,{\mathbb{R}}^2\rightarrow [-\infty ,+\infty)$,
\begin{displaymath} u(x,y):=\begin{cases}-\infty , & {\textrm{when }}\, y\leq 0,\\
1, & {\textrm{when }}\, y>0,\end{cases}\end{displaymath}
which is $2$-quasinearly subharmonic, but surely not quasinearly subharmonic n.s.
\end{remark}

\section[Characterization of harmonic functions via qns-functions]{Characterization of harmonic functions via quasinearly subharmonic functions}

A subharmonic function $u: \Omega \to [-\infty, \infty)$ defined on an open $\Omega \subseteq \mathbb R^N$ is harmonic if and only if the function $-u$ is also subharmonic, \cite{HK}, p.~54. In this section we show that this remains true if one uses quasinearly subharmonic in the narrow sense  functions instead of subharmonic functions.

\begin{proposition}\label{lem1.11}
Let $D$ be a domain in $\mathbb R^N,$ $N\ge 2.$ Then the following statements are equivalent for every $u: D\to[-\infty,\infty).$
\begin{enumerate}
\item The function $u$ is harmonic.
\item There is $K\ge 1$ such that the functions $u$ and $-u$ are $K$-quasinearly subharmonic n.s.
\end{enumerate}
\end{proposition}

\begin{proof}
The implication $(1)\Rightarrow (2)$ is trivial. Suppose statement (2) holds.
Since $u$ and $-u$ are $K$-quasinearly subharmonic n.~s., we have
\[
\frac{K}{\nu_N r^N}\int_{B^N(x, r)} u(y)\,dm_N(y) \leq u(x) \leq \frac{K}{\nu_N r^N}\int_{B^N(x, r)} u(y)\,dm_N(y)
\]
and, consequently,
\begin{equation}\label{eq1.8}
u(x) = \frac{K}{\nu_N r^N}\int_{B^N(x, r)} u(y)\,dm_N(y)
\end{equation}
holds whenever $\overline{B^N(x, r)} \subset D$. Using statement (7) from Proposition~\ref{pr1.5} we see that $u \in \mathfrak{L}_{\mathrm{loc}}^{1}(D)$. It follows from \eqref{eq1.8} for all $x$, $z \in D$ and sufficiently small $r>0$ that
\begin{equation}\label{eq1.9}
|u(x)-u(z)| \leq \frac{K}{\nu_N r^N}\int_{B^N(x, r) \bigtriangleup B^N(z, r)} |u(y)|\,dm_N(y),
\end{equation}
where $B^N(x, r) \bigtriangleup B^N(z, r)$ is the symmetric difference of the balls $B^N(x, r)$ and $B^N(z, r)$. Since
\[
\lim_{x \to z} m_N(B^N(x, r) \bigtriangleup B^N(z, r)) = 0,
\]
the absolute continuity of the Lebesgue integral and the condition $u \in \mathfrak{L}_{loc}^{1}(D)$ imply that $f$ is continuous on $D$. Let $x \in D$ and $u(x)\neq 0$. Equality~\eqref{eq1.8} and continuity of $u$ at the point $x$ imply that $K=1$. Every continuous function $u$ satisfying~\eqref{eq1.8} with $K=1$ for all $B^N(x, r)\subset D$ is harmonic. Statement (1) follows.
\end{proof}

\begin{corollary}
Let $D$ be a domain in $\mathbb R^N,$ $N\ge 2.$ Then a function $u: D\to[-\infty,\infty)$ is harmonic if and only if the functions $u$ and $-u$ are nearly subharmonic.
\end{corollary}

\begin{lemma}\label{lem1.10}
Let $D$ be a domain in $\mathbb R^N$, $N \geq 2$, let $u: D \to [-\infty, \infty)$ be $K_1$-quasinearly subharmonic {n.s.} and let $-u$ be $K_2$-quasinearly subharmonic {n.s.} If there is a point $y_0 \in D$ such that $u(y_0)>0$, then the inequality $K_1 \geq K_2$ holds.
\end{lemma}
\begin{proof}
Let $y_0 \in D$ and $u(y_0)>0$. Then for sufficiently small $r>0$ we have the double inequality
\begin{equation}\label{eq1.6}
\frac{K_2}{\nu_N r^N} \int_{B^N(y_0, r)} u(y)\,dm_N(y) \leq u(y_0) \leq \frac{K_1}{\nu_N r^N} \int_{B^N(y_0, r)} u(y)\,dm_N(y),
\end{equation}
thus
\begin{equation}\label{eq1.7}
\frac{K_2}{\nu_N r^N} \int_{B^N(y_0, r)} u(y)\,dm_N(y) \leq \frac{K_1}{\nu_N r^N} \int_{B^N(y_0, r)} u(y)\,dm_N(y).
\end{equation}
Inequality \eqref{eq1.6}, $u^{+} \in \mathfrak{L}_{loc}^{1}(D)$ and $u(y_0)>0$ imply that
\[
0< \int_{B^N(y_0, r)} u(y)\,dm_N(y) <+\infty.
\]
Now $K_2 \leq K_1$ follows from \eqref{eq1.7}.
\end{proof}

Using this lemma and Proposition~\ref{lem1.10} we obtain the following

\begin{proposition}
Let $D$ be a domain in $\mathbb R^N$, $N \geq 2$. Let $u: D \to [-\infty, \infty)$ be a function such that there are $x_1, x_2\in D$ satisfying the double inequality
\begin{equation}\label{u}
u(x_1)>0>u(x_2).
\end{equation}
Then the function $u$ is harmonic if and only if the functions $u$ and $-u$ are quasinearly subharmonic n.s.
\end{proposition}

\begin{proof}
It suffices to show that $u$ is harmonic if $u$ and $-u$ are quasinearly subharmonic n.s. Let $u$ be $K_1$-quasinearly subharmonic n.s. and $-u$ be $K_2$-quasinearly subharmonic n.s. Then double inequality \eqref{u} and Lemma~\ref{lem1.10} imply the equality $K_1 = K_2.$ Now the harmonicity of $u$ follows from Proposition~\ref{lem1.11}.
\end{proof}

The following example shows that there is $u: D\to (0, \infty)$ such that $u$ and $(-u)$ are quasinearly subharmonic n.s. but $u$ is not harmonic.

\begin{example}
Let $D=\mathbb R^{n}$ and
\begin{displaymath} u(x):=\begin{cases}2 , & {\textrm{when }}\, x = 0,\\
1, & {\textrm{when }}\, x\ne 0.\end{cases}\end{displaymath}
Then $u$ is $2$-quasinearly subharmonic n.s. and $(-u)$ is $1$-quasinearly subharmonic n.s., but $u$ is discontinuous at zero.
\end{example}

\begin{remark}
The above functions $u$ and $(-u)$ are both $2$-quasinearly subharmonic. Thus Proposition~\ref{lem1.11} becames false if we replace the quasinearly subharmonicity n.s. by quasinearly subharmonicity.
\end{remark}

\end{document}